\documentclass[11pt,a4paper]{amsart}
\usepackage{amsmath, amssymb,color, amsthm}
\usepackage[utf8]{inputenc}
\usepackage{bigints}
\usepackage{hyperref, xcolor, tikz}
\usepackage{geometry, enumitem, subcaption}
\usepackage{placeins}
\usepackage{cleveref}
\numberwithin{equation}{section}

\definecolor{goetheblau}{cmyk}{1.00 0.20 00 0.40}
\definecolor{blue}{cmyk}{1.00 0.20 00 0.40}
\definecolor{hellgrau}{cmyk}{0.04 0.04 0.05 0.02}
\definecolor{sandgrau}{cmyk}{0.12 0.09 0.13 0}
\definecolor{dunkelgrau}{cmyk}{0.25 0.25 0.30 0.75}
\definecolor{purple}{cmyk}{0.08 1.00 0.30 0.36}
\definecolor{emorot}{cmyk}{0.04 1.00 0.80 0.07}
\definecolor{red}{cmyk}{0.04 1.00 0.80 0.07}
\definecolor{senfgelb}{cmyk}{0.01 0.25 1.00 0.05}
\definecolor{gruen}{cmyk}{0.62 0.40 0.87 0.09}
\definecolor{magenta}{cmyk}{0.08 0.86 0.12 0.12}
\definecolor{orange}{cmyk}{0 0.70 1.00 0.04}
\definecolor{sonnengelb}{cmyk}{0 0.12 0.95 0}
\definecolor{hellesgruen}{cmyk}{0.40 0.17 0.81 0.07}
\definecolor{lichtblau}{cmyk}{0.80 00 0.06 0.04}

\definecolor{bluediff}{rgb}{0.0, 0.0, 1.0}
\definecolor{reddiff}{rgb}{1.0 0 0.0}

\hypersetup{%
  colorlinks = true,
  linkcolor  = emorot,
  citecolor = hellesgruen
}
\newtheorem{theorem}{Theorem}[section]
\newtheorem{corollary}[theorem]{Corollary}

\newtheorem{lemma}[theorem]{Lemma}

\newtheorem{remark}[theorem]{Remark}

\newcommand{\N}{\mathbb{N}}

\newcommand{\R}{\mathbb{R}}

\usepackage{commath}
\usepackage{esdiff}

\DeclareMathOperator{\EW}{\mathbf{E}}
\DeclareMathOperator{\WS}{\mathbf{P}}

\usepackage{graphicx}
\usepackage{mathrsfs}

\usepackage[]{todonotes}
\hypersetup{pdfauthor={, Jan Lukas Igelbrink, Jasper Ischebeck}%
  pdftitle={},%
  pdfsubject={},%
  pdfkeywords={},%
  pdfproducer={LaTeX},%
  pdfcreator={pdfLaTeX}
}
\usepackage{autonum}
\usetikzlibrary{arrows.meta,bending}
\usetikzlibrary{patterns}
\usepackage{pict2e,color}
\usepackage{placeins}
\usepackage{bigints}
\usepackage{underscore}
\usepackage{soul}  

\usepackage{pgfplots}
\pgfplotsset{compat=1.18}
\usepackage{tikz}
\usepackage{wrapfig}

\title[Ancestral reproductive bias under various sampling schemes]{Ancestral reproductive bias in continuous-time branching trees under various sampling schemes}
\author[J.~L. Igelbrink]{Jan Lukas Igelbrink}
\address{Jan Lukas Igelbrink  \\ Institut f\"ur Mathematik, 
  Johannes Gutenberg-Universit\"at Mainz and Goethe-Universit\"at Frankfurt, Germany.}
\email{igelbrin@math.uni-frankfurt.de}
\author[J. Ischebeck]{Jasper Ischebeck}
\address{Jasper Ischebeck, Institut f\"ur Mathematik, Goethe-Universit\"at Frankfurt, Germany.}
\email{ischebeck@math.uni-frankfurt.de}

\thanks{We thank Anton Wakolbinger for bringing the work \cite{gworg} to our attention. We are grateful to him and also to Matthias Birkner, G\"otz Kersting and Marius Schmidt for stimulating discussions and valuable hints. A substantial part of this work was done during the 2023 seminar week of the Frankfurt probability group in Haus Bergkranz.}

\begin{document}
\begin{abstract}
  Cheek and Johnston \cite{gworg} consider a continuous-time (Bienaymé-)Galton-Watson tree conditioned on being alive at time $T$. They study the reproduction events along the  ancestral lineage of an individual randomly sampled from all those alive at time $T$. We give a short proof of an extension of their main results \cite[Theorems~2.3 and 2.4]{gworg} to the more general case
of Bellman-Harris processes. Our proof also sheds light onto the probabilistic structure of the rate of the reproduction events. A similar method will be applied to explain (i) the different ancestral reproduction bias appearing in work by Geiger \cite{geiger} and (ii) the fact that the sampling rule considered by Chauvin, Rouault and Wakolbinger in \cite[Theorem 1] {CW} leads to a time homogeneous process along the ancestral lineage. 
\end{abstract}
\maketitle
\noindent
\allowdisplaybreaks
\section{Introduction}\label{sec:introduction}

Consider a continuous-time branching process with $N_t$ individuals alive at time $t$, started with one individual at time $0$. At the end of its lifetime, 
an individual is replaced by a random number of independent offspring with distribution $\left(p_k\right)_{k\geq 0}$. When lifetimes of the individuals are i.i.d.\ with an arbitrary distribution $\mu$ on~$\R_+$, the resulting process is called a \textit{Bellman-Harris} process \cite{BH}. In the special case of exponentially distributed lifetimes, 
this process is a continuous-time(Bienaym\'e-) Galton-Watson process, which is also called \emph{one-dimensional continuous-time Markov branching process}, see \cite[Chapter~3]{AN}.
For those processes, Cheek and Johnston \cite{gworg} study the process of reproduction times and family sizes along the ancestral lineage of an individual sampled from all those alive at a given time $T>0$, conditioned on the event $\{N_T>0\}$. We give a short and conceptual probabilistic proof of the main results of \cite{gworg} in the more general Bellman-Harris setting. The core idea of this proof is as follows:

On the event $\left\{ N_T > 0 \right\}$, we assign to the individuals alive at time $T$ independent random variables, which will be called markers, uniformly distributed on $[0,1]$. Then the individual whose marker is largest  constitutes a uniformly distributed random pick from all the individuals alive at time $T$. As we will see, the argument~$s$ of the generating functions that  appear in the analytic arguments of \cite{gworg} corresponds to  the realisation of the largest marker. Sections~\ref{sec:mainresults}~to~\ref{proof2} will be devoted to formulating and proving \Cref{th:1-general}.

Relating to work of Chauvin, Rouault and Wakolbinger \cite{CW}, in Section \ref{secCRW} we will consider the case of potentially dependent but identically and atomless distributed markers and conditioning on one marker taking the prescribed value $s$. In  contrast to the above, in this case one does not observe a time-inhomogeneity along the sampled ancestral lineage. 

In Section \ref{SecGeiger} we will consider a planar embedding of the Bellman-Harris tree conditioned to survive up to time $T$, and analyse the leftmost ancestral lineage among those surviving until time~$T$. Here we follow Geiger~\cite{geiger}, who gave a representation of discrete-time Galton-Watson processes conditioned to survive up to a given number of generations. With this sampling rule we observe a time-inhomogeneity of ancestral reproduction events that is different from the one in \cite{gworg}. 

In Section  \ref{biopersp} we briefly resume the discussion from \cite{gworg} on a possible relation between the ancestral rate bias and the rate of mutations per cell division in embryogenesis, and illustrate the various sampling schemes from a more biological perspective.
\section{Sampling an ancestral line at random}\label{sec:mainresults}
\begin{wrapfigure}{r}{0.5\textwidth}
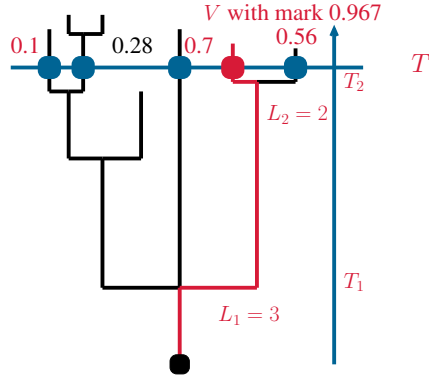

  \centering
   \scalebox{0.4}{ {\pgfkeys{/pgf/fpu/.try=false}%
\ifx\XFigwidth\undefined\dimen1=0pt\else\dimen1\XFigwidth\fi
\divide\dimen1 by 6546
\ifx\XFigheight\undefined\dimen3=0pt\else\dimen3\XFigheight\fi
\divide\dimen3 by 6429
\ifdim\dimen1=0pt\ifdim\dimen3=0pt\dimen1=3946sp\dimen3\dimen1
  \else\dimen1\dimen3\fi\else\ifdim\dimen3=0pt\dimen3\dimen1\fi\fi
\tikzpicture[x=+\dimen1, y=+\dimen3]
{\ifx\XFigu\undefined\catcode`\@11
\def\temp{\alloc@1\dimen\dimendef\insc@unt}\temp\XFigu\catcode`\@12\fi}
\XFigu3946sp
\ifdim\XFigu<0pt\XFigu-\XFigu\fi
\pgfdeclarearrow{
  name = xfiga0,
  parameters = {
    \the\pgfarrowlinewidth \the\pgfarrowlength \the\pgfarrowwidth},
  defaults = {
	  line width=+7.5\XFigu, length=+120\XFigu, width=+60\XFigu},
  setup code = {
    \dimen7 2.15\pgfarrowlength\pgfmathveclen{\the\dimen7}{\the\pgfarrowwidth}
    \dimen7 2\pgfarrowwidth\pgfmathdivide{\pgfmathresult}{\the\dimen7}
    \dimen7 \pgfmathresult\pgfarrowlinewidth
    \pgfarrowssettipend{+\dimen7}
    \pgfarrowssetbackend{+-\pgfarrowlength}
    \dimen9 -0.5\pgfarrowlinewidth
    \pgfarrowssetvisualbackend{+\dimen9}
    \pgfarrowssetlineend{+-0.5\pgfarrowlinewidth}
    \pgfarrowshullpoint{+\dimen7}{+0pt}
    \pgfarrowsupperhullpoint{+-\pgfarrowlength}{+0.5\pgfarrowwidth}
    \pgfarrowssavethe\pgfarrowlinewidth
    \pgfarrowssavethe\pgfarrowlength
    \pgfarrowssavethe\pgfarrowwidth
  },
  drawing code = {\pgfsetdash{}{+0pt}
    \ifdim\pgfarrowlinewidth=\pgflinewidth\else\pgfsetlinewidth{+\pgfarrowlinewidth}\fi
    \pgfpathmoveto{\pgfqpoint{-\pgfarrowlength}{0.5\pgfarrowwidth}}
    \pgfpathlineto{\pgfqpoint{0pt}{0pt}}
    \pgfpathlineto{\pgfqpoint{-\pgfarrowlength}{-0.5\pgfarrowwidth}}
    \pgfusepathqstroke
  }
}
\clip(3321,-6387) rectangle (9867,42);
\tikzset{inner sep=+0pt, outer sep=+0pt}
\pgfsetfillcolor{blue}
\pgftext[base,left,at=\pgfqpointxy{4950}{-750}] {\fontsize{22}{26.4}\usefont{T1}{ptm}{m}{n}0.28}
\pgfsetlinewidth{+7.5\XFigu}
\pgfsetstrokecolor{white}
\pgfsetarrows{[line width=7.5\XFigu]}
\pgfsetarrowsend{xfiga0}
\draw (3525,-150)--(3525,-6300);
\draw (3525,-6225)--(9150,-6225);
\pgfsetlinewidth{+60\XFigu}
\pgfsetstrokecolor{black}
\pgfsetarrowsend{}
\draw (6000,-4425)--(4800,-4425);
\draw (4800,-4425)--(4800,-2400);
\draw (6000,-4425)--(6000,-375);
\pgfsetstrokecolor{blue}
\pgfsetarrows{[line width=45\XFigu, width=105\XFigu]}
\pgfsetarrowsend{xfiga0}
\draw (8400,-5625)--(8400,-300);
\pgfsetstrokecolor{black}
\pgfsetarrowsend{}
\draw (4800,-2400)--(4275,-2400);
\draw (4800,-2400)--(5400,-2400);
\draw (5400,-2400)--(5400,-1350);
\draw (4275,-2400)--(4275,-1350);
\draw (4275,-1350)--(4500,-1350);
\draw (4500,-1350)--(4500,-450);
\draw (4275,-1350)--(3975,-1350);
\draw (3975,-1350)--(3975,-375);
\draw (4500,-450)--(4800,-450);
\draw (4500,-450)--(4275,-450);
\draw (4275,-450)--(4275,-150);
\pgfsetfillcolor{white}
\filldraw (4800,-450)--(4800,-375)--(4800,-150);
\pgfsetstrokecolor{blue}
\draw (3375,-975)--(8850,-975);
\pgfsetcolor{blue}
\filldraw (4125,-1125) [rounded corners=+105\XFigu] rectangle (3825,-825);
\filldraw (4650,-1125) [rounded corners=+105\XFigu] rectangle (4350,-825);
\pgfsetstrokecolor{black}
\draw (7200,-1200)--(7800,-1200);
\draw (7800,-1200)--(7800,-675);
\pgfsetcolor{blue} 
\filldraw (7950,-1125) [rounded corners=+105\XFigu] rectangle (7650,-825);
\filldraw (6150,-1125) [rounded corners=+105\XFigu] rectangle (5850,-825);
\pgfsetcolor{red}
\filldraw (6975,-1125) [rounded corners=+105\XFigu] rectangle (6675,-825);
\draw (6825,-1200)--(6825,-600);
\draw (7200,-1200)--(6825,-1200);
\draw (7200,-4425)--(7200,-1200);
\draw (6000,-4425)--(7200,-4425);
\draw (6000,-5625)--(6000,-4425);
\pgfsetlinewidth{+7.5\XFigu}
\pgfsetstrokecolor{white}
\pgfsetarrows{[line width=7.5\XFigu, width=60\XFigu]}
\pgfsetarrowsend{xfiga0}
\draw (9825,-6375)--(9825,-450);
\pgftext[base,left,at=\pgfqpointxy{8550}{-4425}] {\fontsize{22}{26.4}\usefont{T1}{ptm}{m}{n}$T_1$}
\pgftext[base,left,at=\pgfqpointxy{8550}{-1290}] {\fontsize{22}{26.4}\usefont{T1}{ptm}{m}{n}$T_2$}
\pgfsetfillcolor{blue}
\pgftext[base,left,at=\pgfqpointxy{6075}{-750}] {\fontsize{22}{26.4}\usefont{T1}{ptm}{m}{n}0.7}
\pgftext[base,left,at=\pgfqpointxy{7500}{-600}] {\fontsize{22}{26.4}\usefont{T1}{ptm}{m}{n}0.56}
\pgfsetfillcolor{red}
\pgftext[base,left,at=\pgfqpointxy{7350}{-1800}] {\fontsize{22}{26.4}\usefont{T1}{ptm}{m}{n}$L_2=2$}
\pgftext[base,left,at=\pgfqpointxy{6600}{-4950}] {\fontsize{22}{26.4}\usefont{T1}{ptm}{m}{n}$L_1=3$}
\pgftext[base,left,at=\pgfqpointxy{6375}{-225}] {\fontsize{22}{26.4}\usefont{T1}{ptm}{m}{n}$V$ with mark 0.967}
\pgfsetfillcolor{blue}
\pgftext[base,left,at=\pgfqpointxy{9600}{-1050}] {\fontsize{24}{28.8}\usefont{T1}{ptm}{m}{n}$T$}
\pgftext[base,left,at=\pgfqpointxy{3375}{-750}] {\fontsize{22}{26.4}\usefont{T1}{ptm}{m}{n}0.1}
\pgfsetcolor{black}
\filldraw (6150,-5775) [rounded corners=+105\XFigu] rectangle (5850,-5475);
\endtikzpicture}%
}
  \caption{An example for a realisation of the random variables $S, L_1, L_2, T_1, T_2$ in the sampling regime described in Section~\ref{sec:mainresults}.}\label{fig:1}
\end{wrapfigure}
Recall that to each individual at time $T$, we have associated a uniform marker in $[0,1]$. 
On the event $\{N_T > 0\}$, let the individual $V$ be sampled as described in the Introduction, and let $S$ be its mark. We define the process~$(N_t)_{t\ge 0}$ to be right continuous with left limits. As a consequence, if $T_1$ is the lifetime of the root individual, then $N_{T_1}$ has distribution $\left(p_k\right)_{k\geq 0}$.  Let $J$ be the random number of reproduction events and   ${0<T_1 < T_2 < \cdots < T_J\le T}$ be the random times of reproduction events along the ancestral lineage of $V$. Let $L_1, \ldots, L_J$  be the offspring sizes in these reproduction events and let $0<\tau_1 < \tau_2 < \cdots$ be the random arrival times in a renewal process with interarrival time distribution $\mu$. See Figure~\ref{fig:1} for a sample realisation.\\
Denote by $\WS$ and $\EW$ the probability measure and expectation for $N_0=1$.
  \begin{theorem}\label{th:1-general} For $j \ge 0$, $0  < t_1<\ldots < t_j \le T \in \R$ and $\ell_1,\ldots, \ell_j \in \N$ we have
  \begin{eqnarray}\begin{aligned} \label{fullformula}
  &\WS\left(N_T>0, J=j, \, T_1\in \dif t_1, \ldots T_j\in \dif t_j,\, L_1=\ell_1, \ldots, L_j=\ell_j,\, S\in \dif s\right) \\
  &= \WS \left(\tau_1 \in \dif t_1, \ldots, \tau_j \in \dif t_j, \tau_{j+1} > T\right)    \prod_{i=1}^j   \left( \ell_i p_{\ell_i} \EW\left[ s^{ N_{T-t_i} } \right]^{\ell_i -1 }  \right) \dif{s}.
  \end{aligned}
  \end{eqnarray}
\end{theorem}
\begin{corollary}\label{coro}  When integrated over $s\in (0,1)$, \eqref{fullformula} reveals that 
the process\\$(T_1,L_1),\ldots, (T_J,L_J)$ of reproduction times and offspring sizes along the ancestral lineage of the uniformly chosen individual (conditioned on $\{N_T>0\}$)  is a mixture of  (what could be called)  ``biased compound renewal processes''. 
\end{corollary}
When the lifetime distribution $\mu$ is the exponential distribution with parameter $r$, then $\tau_1, \tau_2,\ldots$ are the points of a rate $r$ Poisson point process. In this case \Cref{coro} together with \eqref{fullformula} becomes a reformulation of the statements of  \cite[Theorems~2.3 and 2.4]{gworg}, and at the same time reveals the probabilistic role of the mixing parameter $s$ in the mixture of biased compound Poisson processes that appear in the ``Cox process representation" of \cite{gworg}.

 Let us write (as in \cite{gworg})  $F_t(s) := \mathbf E[s^{N_t}]$, and abbreviate
\begin{equation}\label{ratebias}
B(t,T, \ell) := \frac 1{1-F_T(0)} \int_0^1 F_{T-t}(s)^{\ell-1}F'_T(s) \dif s.
\end{equation}
\cite[Theorem~2.4]{gworg} (as well as \Cref{th:1-general}) says  that the rate of size $\ell$ reproduction along the uniform ancestral lineage at time $t$ is  $$r\ell p_\ell\,  B(t,T,\ell).$$
This can be obtained from \Cref{coro} by noting that $S$ has density $$\frac{F_T'(s)}{1-F_t(0)}.$$
In this sense the factor $ B(t,T,\ell)$ can be interpreted as an  {\em (ancestral) rate bias}, on top of the classical term $r \ell p_\ell$. Indeed, the factor  $B(t,T,\ell)$ is absent  in trees that are biased with respect to their size at time~$T$. 
Galton-Watson trees of this kind have been investigated (also in the multitype case) by Georgii and Baake \cite[Section~4]{geba}; they are continuous-time analogues of the size-biased trees analysed by Lyons et al. \cite{lpp} and Kurtz et al. \cite{klpp}.

In the critical and supercritical case one can check that, for all fixed $u<T$  and $\ell \in \mathbb N$ one has the convergence  $B(T-u, T, \ell) \to 1$ as $T\to \infty$ because $S$ converges to 1 in probability. 
In the supercritical case this stabilisation along the  sampled ancestral lineage corresponds to the “retrospective viewpoint” that has been taken in \cite{geba} and, in the more general situation of Crump-Mode-Jagers processes, by Jagers and Nerman \cite{jn}.
The choice $\mu = \delta_1$ renders the case of discrete-time Galton-Watson processes, starting with one individual at time $0$ and with reproduction events at times $1,2,\ldots$. Then, with $T=n \in \N$, and $L_1, \ldots, L_{n}$ being the family sizes along the ancestral lineage of the sampled individual $V$, the formula  \eqref{fullformula} specialises to
  \begin{eqnarray}
  \WS\left(N_n>0, \, L_1=\ell_1, \ldots, L_{n}=\ell_{n},\, S\in \dif s\right) 
  =  \left( \prod_{i=1}^{n}   \ell_i p_{\ell_i} \EW\left[ s^{ N_{n-i} } \right]^{\ell_i -1 }  \right) \dif{s}.
  \end{eqnarray} 

\section{Maxima of i.i.d.\ random markers}\label{maxmarkers}
 As a preparation for the short probabilistic proof of \Cref{th:1-general} given in the next section, we recall  the following \mbox{well-know fact:}
  Denote by Unif$[0,1]$ the uniform distribution on the interval $[0,1]$. 
 For $\ell \in \N$, let $\widetilde S$ be the maximum of $\ell$ independent Unif$[0,1]$-distributed random variables $U_1, \ldots, U_\ell$. Then the density of $\widetilde S$ is
\begin{equation}\label{classic}
\WS\left(\widetilde S \in \dif s\right) = \ell s^{\ell-1} \dif s, \quad  0\le s \le 1.
\end{equation}
Indeed, because of exchangeability, 
$$\WS\left(\widetilde S\in \dif s\right) = \ell \WS\left(U_1\in \dif s\right)  \WS\left(U_2<s,\ldots, U_\ell<s\right),$$
which equals the r.h.s.\ of \eqref{classic}. 

The following  lemma specialises to \eqref{classic} when putting $\tilde N \equiv 1$.
 \begin{lemma}\label{Lemma1}
 Let $\widetilde N$ be an $\N_0$-valued random variable, and  $\widetilde N_1, \widetilde N_2, \ldots $ be i.i.d.~copies of  $\tilde N$. Given $\widetilde N_1, \widetilde N_2, \ldots $ let $U_{1,1}, \ldots U_{1,\widetilde N_1},  U_{2,1}, \ldots U_{2,\widetilde N_2}, \ldots$ be independent Unif\,$[0,1]$-distributed random variables, and write
\begin{eqnarray*}
S_k &:=& \max\left\{U_{k,1}, \ldots,  U_{k,\widetilde N_k}\right\}, \quad k=1,2,\ldots\\
S^{(\ell)}&:=& \max\left\{ S_1, \ldots,  S_\ell\right\},\quad \ell \in \N
\end{eqnarray*}
where we put $\max(\emptyset) := -\infty$.
 Then, for all $\ell \in \N$,  the density of $S^{(\ell)}$ is
\begin{equation}\label{gen1}
\WS\left(\widetilde N_1+\ldots+\widetilde N_\ell > 0, \, S^{(\ell)} \in \dif s\right) = \ell\, \EW\left[s^{\widetilde N}\right]^{\ell-1} \WS\left(\widetilde N_1 > 0, \,S_1 \in \dif s\right), \quad  0\le s \le 1.
\end{equation}
 \end{lemma}
\begin{proof}Again because of exchangeability, the l.h.s.\ of \eqref{gen1} equals
\begin{equation}\label{proof1}\ell \WS\left(\widetilde N_1 > 0, \, S_1 \in \dif s\right) \WS\left(S_2 < s, \ldots,  S_\ell < s\right)
\end{equation}
for $s\in [0,1]$.
Since by assumption the $S_k$ are i.i.d.\ copies of $ S_1$, the rightmost factor in \eqref{proof1} equals 
$$\WS\left(S_1< s\right)^{\ell-1} = \EW\left[\WS\left( S_1 < s \mid \widetilde N_1\right)\right]^{\ell-1} = \EW\left[s^{\widetilde N_1}\right]^{\ell-1} = \EW\left[s^{\widetilde N}\right]^{\ell-1}.$$
Hence, \eqref{proof1} equals the r.h.s.\ of \eqref{gen1}, completing the proof of the lemma.
\end{proof}
The following corollary is immediate.
 \begin{corollary}\label{coro2} Let $L$ be an $\N_0$-valued random variable that is independent of all the random variables appearing in \Cref{Lemma1}, with $\WS(L=\ell) = p_\ell$, $\ell \in \N_0$. Then we have for all $\ell \in \N_0$,
 \begin{equation}\label{gen2}
\WS\left(L=\ell,\, \widetilde N_1+\ldots+\widetilde N_\ell > 0, \, S^{(\ell)} \in \dif s\right) = \ell p_\ell\, \EW\left[s^{\widetilde N}\right]^{\ell-1} \WS\left(\widetilde N_1 > 0, \,S_1 \in \dif s\right), \quad  0\le s \le 1.
\end{equation}
 \end{corollary}
 \section{Proof of \texorpdfstring{\Cref{th:1-general}}{Theorem 1}}\label{proof2}
We prove the statement \eqref{th:1-general} by induction over $j$, {\em simultaneously} over all time horizons $T > 0$. 
We write $\mathbf P^T$ for the probability referring to time horizon $T$;  this will  be helpful in the induction step where we  will encounter two different time horizons.\\
For $j=0$, both sides of \eqref{th:1-general} are equal to $ \mu((T,\infty))\, \dif s$.\\
For $j=1$, on the event $\{T_1\in \dif{t_1}\}$, we can directly apply \Cref{coro2} to the markers of the $L_1$ subtrees produced in this event. These subtrees live $T-t_1$ long and thus have sizes distributed as $N_{T-t_1}$.
So the left side of \eqref{th:1-general} equals
\[
 \WS\left( \tau_1 \in \dif{t_1}\right) p_{\ell_1} \ell_1 \EW\left[s^{N_{T-t_1}} \right]^{\ell_1-1} \WS^{T-t_1} \left( T_1 > T-t_1, \, S\in \dif{s} \right),
\]
which is using the $j=0$ case. This is equal to the right hand side of \eqref{th:1-general}.\\ 
Now assume we have proved \eqref{th:1-general} for all time horizons $T'$ with $j-1$ (in place of $j$), for all times~${t_1', \dots, t_{j-1}'\le T'}$, sizes $\ell_1', \dots, \ell_{j-1}' \in \mathbb N$ and $s\in [0,1]$. 
On the event $\left\{ T_1 \in \dif t_1, L_1 = \ell_1 \right\}$ the descendants of the $\ell_1$ siblings in the first branching event form $\ell_1$ independent and identically distributed trees on the time interval $[t_1, T]$. 
Thus, using \Cref{coro2} and setting ${t_1':= t_2-t_1, \ldots, t_{j-1}'= t_j-t_{1}}$, we obtain that the left hand side of \eqref{th:1-general} equals 
\begin{equation}
\begin{split}
     \WS\left( \tau_1 \in \dif{t_1}\right) p_{\ell_1} \ell_1 &\EW\left[s^{N_{T-t_1}} \right]^{\ell_1-1} \\&\cdot \WS^{T-t_1} \left( J=j-1, \, T_1 \in \dif{t_1}', \ldots, T_{j-1} \in \dif{t_{j-1}}', \, N_{T-t_1} >0, \, S\in \dif{s} \right).
\end{split}
\end{equation}
By the induction assumption, this is equal to 
\begin{equation}
    \WS\left( \tau_1 \in \dif{t_1} \right) p_{\ell_1} \ell_1 \WS \left(  \tau_1' \in \dif{t_1}', \ldots, \tau_{j-1}' \in \dif{t_{j-1}'}, \, \tau_j' \geq T-t_1 \right)   \prod_{i=2}^j \left(\ell_i p_{\ell_i} \EW\left[s^{N_{T-t_i}} \right]^{\ell_i-1} \right),
    \label{eq:indfolth21}
\end{equation}
 where $\left( \tau_1', \tau_2', \ldots \right) $ have the same distribution as $\left(\tau_1, \tau_2 ,\ldots \right).$ Obviously \eqref{eq:indfolth21} equals the r.h.s of \eqref{th:1-general}. This completes the induction step and concludes the proof. \hfill $\qed$ 
    \section{Conditioning on a marker value}\label{secCRW}
Chauvin, Rouault and Wakolbinger  \cite{CW} consider a Markov process with an atomless transition probability indexed by a continuous-time Galton-Watson-tree and  they then condition on an individual at time $T$ to be at a given location. 

\begin{wrapfigure}[18]{r}{0.5\textwidth}
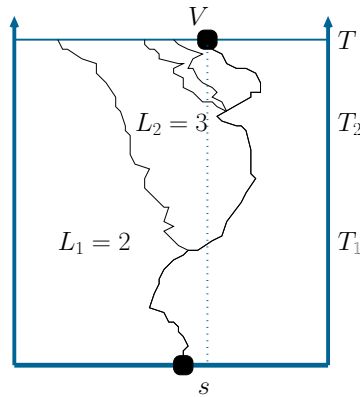

  \centering
   \scalebox{0.4}{ {\pgfkeys{/pgf/fpu/.try=false}%
\ifx\XFigwidth\undefined\dimen1=0pt\else\dimen1\XFigwidth\fi
\divide\dimen1 by 5709
\ifx\XFigheight\undefined\dimen3=0pt\else\dimen3\XFigheight\fi
\divide\dimen3 by 6384
\ifdim\dimen1=0pt\ifdim\dimen3=0pt\dimen1=3946sp\dimen3\dimen1
  \else\dimen1\dimen3\fi\else\ifdim\dimen3=0pt\dimen3\dimen1\fi\fi
\tikzpicture[x=+\dimen1, y=+\dimen3]
{\ifx\XFigu\undefined\catcode`\@11
\def\temp{\alloc@1\dimen\dimendef\insc@unt}\temp\XFigu\catcode`\@12\fi}
\XFigu3946sp
\ifdim\XFigu<0pt\XFigu-\XFigu\fi
\pgfdeclarearrow{
  name = xfiga0,
  parameters = {
    \the\pgfarrowlinewidth \the\pgfarrowlength \the\pgfarrowwidth},
  defaults = {
	  line width=+7.5\XFigu, length=+120\XFigu, width=+60\XFigu},
  setup code = {
    \dimen7 2.15\pgfarrowlength\pgfmathveclen{\the\dimen7}{\the\pgfarrowwidth}
    \dimen7 2\pgfarrowwidth\pgfmathdivide{\pgfmathresult}{\the\dimen7}
    \dimen7 \pgfmathresult\pgfarrowlinewidth
    \pgfarrowssettipend{+\dimen7}
    \pgfarrowssetbackend{+-\pgfarrowlength}
    \dimen9 -0.5\pgfarrowlinewidth
    \pgfarrowssetvisualbackend{+\dimen9}
    \pgfarrowssetlineend{+-0.5\pgfarrowlinewidth}
    \pgfarrowshullpoint{+\dimen7}{+0pt}
    \pgfarrowsupperhullpoint{+-\pgfarrowlength}{+0.5\pgfarrowwidth}
    \pgfarrowssavethe\pgfarrowlinewidth
    \pgfarrowssavethe\pgfarrowlength
    \pgfarrowssavethe\pgfarrowwidth
  },
  drawing code = {\pgfsetdash{}{+0pt}
    \ifdim\pgfarrowlinewidth=\pgflinewidth\else\pgfsetlinewidth{+\pgfarrowlinewidth}\fi
    \pgfpathmoveto{\pgfqpoint{-\pgfarrowlength}{0.5\pgfarrowwidth}}
    \pgfpathlineto{\pgfqpoint{0pt}{0pt}}
    \pgfpathlineto{\pgfqpoint{-\pgfarrowlength}{-0.5\pgfarrowwidth}}
    \pgfusepathqstroke
  }
}
\clip(3483,-6417) rectangle (9192,-33);
\tikzset{inner sep=+0pt, outer sep=+0pt}
\pgfsetfillcolor{red}
\pgftext[base,left,at=\pgfqpointxy{8850}{-3975}] {\fontsize{24}{28.8}\usefont{T1}{ptm}{m}{n}$T_1$}
\pgfsetlinewidth{+7.5\XFigu}
\pgfsetstrokecolor{white}
\pgfsetarrows{[line width=7.5\XFigu]}
\pgfsetarrowsend{xfiga0}
\draw (3525,-150)--(3525,-6300);
\pgfsetlinewidth{+60\XFigu}
\pgfsetstrokecolor{blue}
\pgfsetarrows{[line width=45\XFigu, width=105\XFigu]}
\draw (3825,-5775)--(3825,-300);
\pgfsetarrowsend{}
\pgfsetlinewidth{+15\XFigu}
\pgfsetstrokecolor{black}
\draw (6450,-5775)--(6450,-5550)--(6525,-5475)--(6450,-5325)--(6375,-5250)--(6375,-5175)
  --(6300,-5100)--(6225,-5100)--(6150,-4950)--(5925,-4875)--(6000,-4650)--(6075,-4500)
  --(6075,-4425)--(6150,-4275)--(6300,-4125)--(6450,-4050)--(6525,-3975)--(6675,-3975)
  --(6900,-3825)--(7050,-3750)--(7125,-3525)--(7350,-3225)--(7425,-3075)--(7575,-2850)
  --(7500,-2775)--(7500,-2550)--(7500,-2400)--(7425,-2100)--(7350,-2025)--(7275,-2025)
  --(7050,-1950)--(6975,-1875)--(7500,-1575)--(7500,-1500)--(7650,-1500)--(7650,-1350)
  --(7575,-1200)--(7425,-1125)--(7275,-1125)--(7125,-1050)--(7050,-900)--(6900,-825)
  --(6825,-750);
\pgfsetlinewidth{+60\XFigu}
\pgfsetstrokecolor{blue}
\pgfsetarrowsend{xfiga0}
\draw (8700,-5775)--(8700,-300);
\pgfsetlinewidth{+75\XFigu}
\pgfsetarrowsend{}
\draw (3825,-5775)--(8700,-5775);
\pgfsetlinewidth{+15\XFigu}
\pgfsetstrokecolor{black}
\draw (6525,-3975)--(6375,-3900)--(6300,-3825)--(6150,-3675)--(6300,-3600)--(6150,-3525)
  --(6225,-3300)--(6150,-3300)--(6000,-3150)--(5850,-3000)--(5850,-2925)--(5850,-2850)
  --(5925,-2700)--(5700,-2625)--(5625,-2475)--(5625,-2400)--(5625,-2250)--(5625,-2175)
  --(5550,-2100)--(5550,-2025)--(5475,-1800)--(5400,-1725)--(5400,-1575)--(5250,-1350)
  --(5100,-1200)--(5100,-1050)--(4875,-975);
\pgfsetlinewidth{+7.5\XFigu}
\pgfsetcolor{black}
\filldraw (6600,-5925) [rounded corners=+105\XFigu] rectangle (6300,-5625);
\pgfsetlinewidth{+30\XFigu}
\pgfsetstrokecolor{blue}
\draw (3825,-675)--(8700,-675);
\pgfsetlinewidth{+15\XFigu}
\pgfsetstrokecolor{black}
\draw (6450,-5775)--(6450,-5550)--(6525,-5475)--(6450,-5325)--(6375,-5250)--(6375,-5175)
  --(6300,-5100)--(6225,-5100)--(6150,-4950)--(5925,-4875)--(6000,-4650)--(6075,-4500)
  --(6075,-4425)--(6150,-4275)--(6300,-4125)--(6450,-4050)--(6525,-3975)--(6675,-3975)
  --(6900,-3825)--(7050,-3750)--(7125,-3525)--(7350,-3225)--(7425,-3075)--(7575,-2850)
  --(7500,-2775)--(7500,-2550)--(7500,-2400)--(7425,-2100)--(7350,-2025)--(7275,-2025)
  --(7050,-1950)--(6975,-1875)--(7500,-1575)--(7500,-1500)--(7650,-1500)--(7650,-1350)
  --(7575,-1200)--(7425,-1125)--(7275,-1125)--(7125,-1050)--(7050,-900)--(6900,-825)
  --(6825,-750);
\pgfsetlinewidth{+7.5\XFigu}
\filldraw (6975,-825) [rounded corners=+105\XFigu] rectangle (6675,-525);
\pgfsetlinewidth{+30\XFigu}
\pgfsetstrokecolor{blue}
\pgfsetdash{{+15\XFigu}{+90\XFigu}}{+15\XFigu}
\draw (6825,-675)--(6825,-5775);
\pgfsetlinewidth{+15\XFigu}
\pgfsetdash{}{+0pt}
\pgfsetstrokecolor{black}
\draw (7125,-1800)--(7050,-1650)--(6900,-1575)--(6975,-1500)--(6900,-1425)--(6750,-1350)
  --(6675,-1275)--(6600,-1125)--(6525,-1125)--(6600,-975)--(6750,-900)--(6525,-825)
  --(6375,-750)--(6300,-675);
\draw (4875,-975)--(4650,-825)--(4575,-825)--(4500,-675);
\draw (7125,-1800)--(6975,-1725)--(6900,-1650)--(6750,-1650)--(6675,-1575)--(6600,-1500)
  --(6450,-1425)--(6450,-1350)--(6450,-1275)--(6300,-1200)--(6300,-1125)--(6450,-1050)
  --(6375,-1050)--(6075,-975)--(6075,-900)--(6000,-825)--(5925,-825)--(5850,-750)
  --(5850,-675);
\pgfsetlinewidth{+7.5\XFigu}
\pgfsetstrokecolor{white}
\pgfsetdash{}{+0pt}
\pgfsetarrows{[line width=7.5\XFigu, width=60\XFigu]}
\pgfsetarrowsend{xfiga0}
\draw (3525,-75)--(9150,-75);
\draw (3525,-6375)--(9150,-6375);
\pgfsetfillcolor{red}
\pgftext[base,left,at=\pgfqpointxy{4500}{-3975}] {\fontsize{24}{28.8}\usefont{T1}{ptm}{m}{n}$L_1=2$}
\pgfsetfillcolor{blue}
\pgftext[base,left,at=\pgfqpointxy{6675}{-6225}] {\fontsize{24}{28.8}\usefont{T1}{ptm}{m}{n}$s$}
\pgfsetfillcolor{black}
\pgftext[base,left,at=\pgfqpointxy{6525}{-450}] {\fontsize{24}{28.8}\usefont{T1}{ptm}{m}{n}$V$}
\pgfsetfillcolor{blue}
\pgftext[base,left,at=\pgfqpointxy{8850}{-825}] {\fontsize{24}{28.8}\usefont{T1}{ptm}{m}{n}$T$}
\pgfsetfillcolor{red}
\pgftext[base,left,at=\pgfqpointxy{5700}{-2100}] {\fontsize{24}{28.8}\usefont{T1}{ptm}{m}{n}$L_2=3$}
\pgftext[base,left,at=\pgfqpointxy{8850}{-2100}] {\fontsize{24}{28.8}\usefont{T1}{ptm}{m}{n}$T_2$}
\draw (9150,-6225)--(9150,-300);
\endtikzpicture}%
}
  \caption{An example for a realisation of the random variables $L_1, L_2, T_1, T_2$ in the sampling regime described in Section~\ref{secCRW}.}\label{fig:2}
\end{wrapfigure}

 To relate this to the framework described in the {\hypersetup{hidelinks}\hyperref[sec:introduction]{Introduction}}, we assume that each individual alive at time $T$ in the Bellmann-Harris tree carries a marker in some standard Borel space $E$
 and these random marks have the following properties:
\begin{enumerate}
    \item[(M1)] \label{marker:property1}Their marginal distributions   (denoted by $\nu$)   are identical and do not depend  on the reproduction events
    \item[(M2)]\label{marker:property2} A.s.\ no pair of marks is equal. 
\end{enumerate}

Think for example of branching Brownian motion: The positions
of the different particles clearly depend on each other via the genealogy, however,  at time $t$ the marginal distribution of the position of each particle is a centered Gaussian random variable with variance $t$, irrespective of its past genealogical events in the underlying continuous-time Galton-Watson tree. Thus \hyperref[marker:property1]{(M1)}, is fulfilled. Since two correlated Gaussian random variables are a.s. not equal if the correlation coefficient is not equal to one,  \hyperref[marker:property2]{(M2)}  is also fulfilled. 

  We now condition on $\left\{N_T>0\right\}$ and, for given $s\in E$, on one of the $N_T$ individuals 
  having marker value $s$.
Remember the previous notation:
  Denote by $V$ the individual having marker $s$. 
Let $J$ be the random number of reproduction events along the ancestral lineage of $V$ and $0<T_1 < T_2 < \cdots < T_J< T$ be the random times of these reproduction events. Let $L_1, \ldots, L_J$  be the offspring sizes in these reproduction events and let $0<\tau_1 < \tau_2 < \cdots$ be the random arrival times in a renewal process with interarrival time distribution $\mu$. Figure~\ref{fig:2} depicts a sample realisation. \\
The following Theorem generalises (part of) \cite[Theorem~2]{CW} to general lifetime time distributions.
  \begin{theorem}\label{th:CW} For $j \ge 0$, $0  < t_1<\ldots < t_j < T$ and $\ell_1,\ldots, \ell_j \in \N$ we have for $\nu$-almost all $s$
  \begin{eqnarray}\begin{aligned} \label{fullformula2cw}
  &\WS\left(\left.J=j, \, T_1\in \dif t_1, \ldots, T_j\in \dif t_j,\, L_1=\ell_1, \ldots, L_j=\ell_j \right\vert N_T>0, \exists\, \mathrm{ marker }\in \dif s  \right) \\
  &= \frac1{\EW[N_T]}\WS \left(\tau_1 \in \dif t_1, \ldots, \tau_j \in \dif t_j, \tau_{j+1} \ge T\right)     \prod_{i=1}^j \ell_{i} p_{\ell_i}.
  \end{aligned}
  \end{eqnarray}
\end{theorem}
\begin{proof}
Because of properties \hyperref[marker:property1]{(M1)}, \hyperref[marker:property2]{(M2)}
we have
$$\WS(N_T > 0, \exists \, \mathrm{ marker }\in \dif s) = \EW[N_T]\nu(\dif s), \quad s \in E.$$
Hence \eqref{fullformula2cw} is equivalent to
 \begin{eqnarray}\begin{aligned} \label{fullformula3cw}
  &\WS\left(J=j, \, T_1\in \dif t_1, \ldots, T_j\in \dif t_j,\, L_1=\ell_1, \ldots, L_j=\ell_j,  N_T>0, \exists\, \mathrm{ marker }\in \dif s  \right) \\
  &=\WS \left(\tau_1 \in \dif t_1, \ldots, \tau_j \in \dif t_j, \tau_{j+1} \ge T\right)     \prod_{i=1}^j \ell_{i} p_{\ell_i} \, \nu(\dif s).
  \end{aligned}
  \end{eqnarray}
  As in the proof of \Cref{th:1-general} we prove the statement
\eqref{fullformula3cw} by induction over $j$, {\em simultaneously} over all time horizons $T > 0$.
As before we write $\mathbf P^T$ for the probability referring to time horizon $T$.
For $j=0$ the statement is true, since 
$$\WS^T(J=0,  N_T>0, \exists\, \mathrm{ marker }\in \dif s) = \WS\left(\tau_1 \le T\right)\, \nu(ds).$$
Assume we have proved \eqref{fullformula3cw} for all  time horizons $T'$ with $j-1$ (in place of $j$), for all times~${t_1', \dots, t_{j-1}'\le T'}$, sizes $\ell_1', \dots, \ell_{j-1}' \in \mathbb N$ and marker distributions with the same marginal $\nu$ that satisfy conditions \hyperref[marker:property1]{(M1)}, \hyperref[marker:property2]{(M2)}.
    Turning to \eqref{fullformula3cw} as it stands, we note that   on  $\{T_1 = t_1, L_1=\ell_1\}$,   the descendants of the $\ell_1$ siblings in the first branching event form $\ell_1$ independent and identically distributed trees on the time interval $[t_1, T]$. Let $\mathcal U_k,\, k=1,\dots, \ell_1$, be the set of markers of the individuals at time~$T$ that descend from the $k$-th sibling. By randomly permuting these $\ell_1$ siblings, we can assume that the set-valued random variables $\mathcal U_k,\,k=1,\ldots, \ell_1$, are exchangeable.   
    Note that the markers in each $\mathcal U_k$ satisfy conditions \hyperref[marker:property1]{(M1)}, \hyperref[marker:property2]{(M2)}.   Because the markers are a.s.\ pairwise different by assumption, the marker $s$ belongs to at most one of those~$\mathcal U_k$, so
\begin{equation}
    \mathbf 1_{\{ \exists\, \mathrm{ marker }\in \dif s  \}} = \sum_{k=1}^{\ell_1} \mathbf 1_{\{\mathcal U_k \cap \dif s\ne \emptyset\}} \, \mbox{ a.s.}
\end{equation}
Note that for the sake of intuition we use a differential notation for what formally is an (integral) equality for the distribution of the random point measure formed by the individuals’ markers, which by assumption \hyperref[marker:property2]{(M2)} can be seen as a random set of points.\\
 Putting $t_1' := t_2-t_1, \ldots, t_{j-1}' := t_j-t_1$ we thus infer, using the branching property of the Bellman-Harris tree,  that the left hand side of \eqref{fullformula3cw} equals
\begin{equation}\label{fullformula4cw}
\begin{split}
    &\WS(\tau_1 \in \dif t_1)p_{\ell_1} \ell_1\\
    \cdot&\WS^{T-t_1}\left(J=j-1, \,T_1 \in \dif t_1',\ldots, T_{j-1} \in \dif t_{j-1}', L_1=\ell_2,\ldots, L_{j-1}=\ell_j, N_{T-t_1} > 0, \exists\, {\rm mark} \in \dif s\right) .
\end{split}
\end{equation}
By the induction assumption this is equal to 
\begin{equation}\label{fullformula5cw}\WS(\tau_1 \in \dif t_1)p_{\ell_1} \ell_1\WS \left(\tau_1' \in \dif t_1', \ldots, \tau_{j-1}' \in \dif t_{j-1}', \tau_{j}' \ge T-t_1\right)     \prod_{i=2}^j \ell_{i} p_{\ell_i} \, \nu(\dif s),
\end{equation}
where $(\tau_1', \tau_2', \ldots)$ have the same distribution as $(\tau_1, \tau_2,\ldots)$.
Obviously \eqref{fullformula5cw} equals the r.h.s. of \eqref{fullformula3cw}, which completes the induction step and concludes the proof. 
\end{proof}
 \begin{remark}
  If $\mu$ is the exponential distribution with parameter $r$, then $\tau_1, \tau_2, \ldots$ are again the points of a rate $r$ Poisson point process and \eqref{fullformula2cw} implies that reproduction events along the ancestral lineage of $V$ happen according to a time-homogeneous Poisson process with rate $r \sum_{\ell }\ell p_{\ell}$. This corresponds to the description of the events along the ancestral line of $V$ given in \cite[Theorem~1]{CW}.
\end{remark}
  \section{Sampling the left-most ancestral lineage}\label{SecGeiger}
   \begin{wrapfigure}[20]{r}{0.5\textwidth}
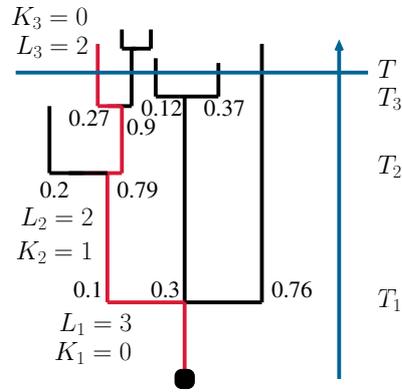

  \centering
   \scalebox{0.4}{ {\pgfkeys{/pgf/fpu/.try=false}%
\ifx\XFigwidth\undefined\dimen1=0pt\else\dimen1\XFigwidth\fi
\divide\dimen1 by 6807
\ifx\XFigheight\undefined\dimen3=0pt\else\dimen3\XFigheight\fi
\divide\dimen3 by 6576
\ifdim\dimen1=0pt\ifdim\dimen3=0pt\dimen1=3946sp\dimen3\dimen1
  \else\dimen1\dimen3\fi\else\ifdim\dimen3=0pt\dimen3\dimen1\fi\fi
\tikzpicture[x=+\dimen1, y=+\dimen3]
{\ifx\XFigu\undefined\catcode`\@11
\def\temp{\alloc@1\dimen\dimendef\insc@unt}\temp\XFigu\catcode`\@12\fi}
\XFigu3946sp
\ifdim\XFigu<0pt\XFigu-\XFigu\fi
\pgfdeclarearrow{
  name = xfiga0,
  parameters = {
    \the\pgfarrowlinewidth \the\pgfarrowlength \the\pgfarrowwidth},
  defaults = {
	  line width=+7.5\XFigu, length=+120\XFigu, width=+60\XFigu},
  setup code = {
    \dimen7 2.15\pgfarrowlength\pgfmathveclen{\the\dimen7}{\the\pgfarrowwidth}
    \dimen7 2\pgfarrowwidth\pgfmathdivide{\pgfmathresult}{\the\dimen7}
    \dimen7 \pgfmathresult\pgfarrowlinewidth
    \pgfarrowssettipend{+\dimen7}
    \pgfarrowssetbackend{+-\pgfarrowlength}
    \dimen9 -0.5\pgfarrowlinewidth
    \pgfarrowssetvisualbackend{+\dimen9}
    \pgfarrowssetlineend{+-0.5\pgfarrowlinewidth}
    \pgfarrowshullpoint{+\dimen7}{+0pt}
    \pgfarrowsupperhullpoint{+-\pgfarrowlength}{+0.5\pgfarrowwidth}
    \pgfarrowssavethe\pgfarrowlinewidth
    \pgfarrowssavethe\pgfarrowlength
    \pgfarrowssavethe\pgfarrowwidth
  },
  drawing code = {\pgfsetdash{}{+0pt}
    \ifdim\pgfarrowlinewidth=\pgflinewidth\else\pgfsetlinewidth{+\pgfarrowlinewidth}\fi
    \pgfpathmoveto{\pgfqpoint{-\pgfarrowlength}{0.5\pgfarrowwidth}}
    \pgfpathlineto{\pgfqpoint{0pt}{0pt}}
    \pgfpathlineto{\pgfqpoint{-\pgfarrowlength}{-0.5\pgfarrowwidth}}
    \pgfusepathqstroke
  }
}
\clip(3285,-6312) rectangle (10092,264);
\tikzset{inner sep=+0pt, outer sep=+0pt}
\pgfsetfillcolor{red}
\pgftext[base,left,at=\pgfqpointxy{9000}{-4500}] {\fontsize{24}{28.8}\usefont{T1}{ptm}{m}{n}$T_1$}
\pgfsetlinewidth{+7.5\XFigu}
\pgfsetstrokecolor{white}
\pgfsetarrows{[line width=7.5\XFigu]}
\pgfsetarrowsend{xfiga0}
\draw (3525,-6225)--(9150,-6225);
\pgfsetlinewidth{+60\XFigu}
\pgfsetstrokecolor{black}
\pgfsetarrowsend{}
\draw (6000,-4425)--(7200,-4425);
\pgfsetstrokecolor{blue}
\pgfsetarrows{[line width=45\XFigu, width=105\XFigu]}
\pgfsetarrowsend{xfiga0}
\draw (8400,-5625)--(8400,-300);
\pgfsetstrokecolor{black}
\pgfsetarrowsend{}
\draw (6000,-4425)--(6000,-1200);
\draw (6525,-1200)--(6525,-675);
\draw (7200,-4425)--(7200,-375);
\draw (4425,-2400)--(3900,-2400);
\draw (5175,-1350)--(5175,-450);
\draw (5175,-450)--(5475,-450);
\draw (5250,-450)--(5025,-450);
\draw (5925,-1200)--(5550,-1200);
\draw (5550,-1200)--(5550,-600);
\pgfsetfillcolor{white}
\filldraw (5025,-450)--(5025,-375)--(5025,-150);
\draw (5475,-450)--(5475,-150);
\pgfsetlinewidth{+7.5\XFigu}
\pgfsetstrokecolor{white}
\pgfsetarrows{[line width=7.5\XFigu, width=60\XFigu]}
\pgfsetarrowsend{xfiga0}
\draw (3525,-150)--(3525,-6300);
\pgfsetlinewidth{+60\XFigu}
\pgfsetstrokecolor{black}
\pgfsetarrowsend{}
\draw (3900,-2400)--(3900,-1350);
\draw (4950,-1350)--(5175,-1350);
\draw (5925,-1200)--(6525,-1200);
\pgfsetstrokecolor{red}
\draw (6000,-5625)--(6000,-4425);
\draw (6000,-4425)--(4800,-4425);
\draw (4800,-4425)--(4800,-2400);
\draw (4425,-2400)--(5025,-2400);
\draw (5025,-2400)--(5025,-1350);
\draw (4950,-1350)--(4650,-1350);
\draw (4650,-1350)--(4650,-375);
\pgfsetstrokecolor{black}
\draw (4800,-2400)--(4200,-2400);
\pgfsetstrokecolor{red}
\draw (5025,-1350)--(4800,-1350);
\pgfsetlinewidth{+7.5\XFigu}
\pgfsetstrokecolor{white}
\pgfsetarrowsend{xfiga0}
\draw (10050,-6075)--(10050,-150);
\draw (9150,-6225)--(9150,-300);
\pgfsetlinewidth{+60\XFigu}
\pgfsetstrokecolor{blue}
\pgfsetarrowsend{}
\draw (3375,-825)--(8850,-825);
\pgfsetfillcolor{blue}
\pgftext[base,left,at=\pgfqpointxy{5100}{-1800}] {\fontsize{22}{26.4}\usefont{T1}{ptm}{m}{n}0.9}
\pgftext[base,left,at=\pgfqpointxy{6300}{-1500}] {\fontsize{22}{26.4}\usefont{T1}{ptm}{m}{n}0.37}
\pgftext[base,left,at=\pgfqpointxy{4275}{-4350}] {\fontsize{22}{26.4}\usefont{T1}{ptm}{m}{n}0.1}
\pgftext[base,left,at=\pgfqpointxy{5475}{-4350}] {\fontsize{22}{26.4}\usefont{T1}{ptm}{m}{n}0.3}
\pgftext[base,left,at=\pgfqpointxy{7350}{-4350}] {\fontsize{22}{26.4}\usefont{T1}{ptm}{m}{n}0.76}
\pgftext[base,left,at=\pgfqpointxy{3750}{-2775}] {\fontsize{22}{26.4}\usefont{T1}{ptm}{m}{n}0.2}
\pgfsetfillcolor{red}
\pgftext[base,left,at=\pgfqpointxy{3450}{-3225}] {\fontsize{24}{28.8}\usefont{T1}{ptm}{m}{n}$L_2=2$}
\pgftext[base,left,at=\pgfqpointxy{3975}{-5325}] {\fontsize{24}{28.8}\usefont{T1}{ptm}{m}{n}$K_1=0$}
\pgfsetfillcolor{blue}
\pgftext[base,left,at=\pgfqpointxy{4200}{-1650}] {\fontsize{22}{26.4}\usefont{T1}{ptm}{m}{n}0.27}
\pgftext[base,left,at=\pgfqpointxy{5325}{-1500}] {\fontsize{22}{26.4}\usefont{T1}{ptm}{m}{n}0.12}
\pgftext[base,left,at=\pgfqpointxy{4950}{-2775}] {\fontsize{22}{26.4}\usefont{T1}{ptm}{m}{n}0.79}
\pgfsetfillcolor{red}
\pgftext[base,left,at=\pgfqpointxy{3375}{-3750}] {\fontsize{24}{28.8}\usefont{T1}{ptm}{m}{n}$K_2=1$}
\pgftext[base,left,at=\pgfqpointxy{4050}{-4875}] {\fontsize{24}{28.8}\usefont{T1}{ptm}{m}{n}$L_1=3$}
\pgftext[base,left,at=\pgfqpointxy{3300}{-75}] {\fontsize{24}{28.8}\usefont{T1}{ptm}{m}{n}$K_3=0$}
\pgftext[base,left,at=\pgfqpointxy{3375}{-525}] {\fontsize{24}{28.8}\usefont{T1}{ptm}{m}{n}$L_3=2$}
\pgfsetfillcolor{blue}
\pgftext[base,left,at=\pgfqpointxy{9000}{-900}] {\fontsize{24}{28.8}\usefont{T1}{ptm}{m}{n}$T$}
\pgfsetfillcolor{red}
\pgftext[base,left,at=\pgfqpointxy{9000}{-1350}] {\fontsize{24}{28.8}\usefont{T1}{ptm}{m}{n}$T_3$}
\pgftext[base,left,at=\pgfqpointxy{9000}{-2400}] {\fontsize{24}{28.8}\usefont{T1}{ptm}{m}{n}$T_2$}
\pgfsetlinewidth{+7.5\XFigu}
\pgfsetcolor{black}
\filldraw (6150,-5775) [rounded corners=+105\XFigu] rectangle (5850,-5475);
\endtikzpicture}%
}
\caption{An example for a realisation of markers and random variables $L_1, L_2,K_1, k_2, T_1, T_2$ in the sampling regime described in Section~\ref{SecGeiger}.}\label{fig:3}  
\end{wrapfigure}
We now aim to obtain results about what Geiger \cite{geiger} calls the leftmost surviving ancestral lineage in a planar embedding of the tree: At any reproduction event we assign independent uniformly on $[0,1]$ distributed markers to all children. An individual can now be uniquely determined by the markers along its ancestral lineage. On the event $\{N_T > 0\}$, let $V$ be the individual whose markers along the entire ancestral lineage comes first in the lexicographic ordering.
Let $J$ be the random number of reproduction events and ${0 < T_1 < T_2 < \cdots < T_J \le T}$ be the random times of reproduction events along the ancestral lineage of $V$. Let $L_1, \ldots, L_J$  be the offspring sizes in these reproduction events and let $0<\tau_1 < \tau_2 < \cdots$ be the random arrival times in a renewal process with interarrival time distribution $\mu$. Denote by $K_i$ the number of siblings born at reproduction event number $i$ along the ancestral lineage of $V$ which have a lower lexicographic order than $V$ and whose descendants hence die out before time $T$. Figure~\ref{fig:3} shows a realisation for this sampling rule.
  \begin{theorem}\label{th:2-general} 
  For $j \ge 0$, $0  < t_1<\ldots < t_j < T, \,  \ell_1,\ldots, \ell_j \in \N$ and $k_i \in \left\{ 1,\ldots, \ell_i-1 \right\}$ we have
  \begin{equation}\begin{split} \label{fullformula2}
  &\WS\left(N_T>0, J=j, \, T_1\in \dif t_1, \ldots , T_j\in \dif t_j,\, L_1=\ell_1, \ldots, L_j=\ell_j, K_1=k_1,\ldots, K_j=k_j\right) \\
  &= \WS \left(\tau_1 \in \dif t_1, \ldots, \tau_j \in \dif t_j, \tau_{j+1} \ge T\right)    \prod_{i=1}^j \left(  p_{\ell_i}\WS\left( N_{T-t_i} =0\right)^{k_i} 
  \right) .
  \end{split}
  \end{equation}
\end{theorem}
\begin{proof}
The proof of the theorem works in analogy to the one of \Cref{th:1-general}, but using following analogue of \Cref{Lemma1}.
\end{proof}
\begin{lemma}\label{LemmaGeiger}
     Let $\widetilde N$ be an $\N_0$-valued random variable, and  $\widetilde N_1, \widetilde N_2, \ldots $ be i.i.d.~copies of  $\tilde N$. Given $\widetilde N_1, \widetilde N_2, \ldots $ let $U_{1},  U_{2}, \ldots$ be independent Unif\,$[0,1]$-distributed random variables, and write
    \begin{eqnarray*}
        S^{(\ell)} &:=& \min\left\{U_{k} \mid\widetilde N_k \ge 1, k=1,\ldots,\ell \right\}, \\
        K^{(\ell)} &:=& \left|\left\{U_k \mid U_k < S^{(\ell)}, k=1,\ldots,\ell \right\}\right|
    \end{eqnarray*}
    where we put $\min(\emptyset) := {+}\infty$.
     Then, for all $k < \ell \in \N$ we have
    \begin{equation}\label{gen2}
    \WS\left(\widetilde N_1+\ldots+\widetilde N_\ell > 0, \, K^{(\ell)} = k \right)
    = \WS\left(\widetilde N=0\right)^{k} \WS\left(\widetilde N > 0\right).
    \end{equation}
 \end{lemma}
 \begin{proof}
     Because $S^{(\ell)}$ and $K^{(\ell)}$ do not depend on the order of $U_1, \dots, U_\ell$, we can use
     exchangeability to assume that $U_1<U_2<\dots<U_\ell$. 
     For $K^{(\ell)}$ to be $k$, $S^{(\ell)}$ has then to be $U_{k+1}$. This is exactly the case if $\widetilde N_1, \dots, \widetilde N_k = 0$ and $\widetilde N_{k+1}>0$.
 \end{proof}
 \section{Biological perspectives}\label{biopersp}
 Cheek and Johnston \cite[Section 5] {gworg} discuss recent studies (\cite{park}, \cite{coorens}) which suggest that certain mutation rates  are elevated for the earliest cell divisions in embryogenesis.  Under the assumptions that (1) cell division times vary and (2) mutations arise not only {\em at} but also {\em between} cell divisions,  Cheek and Johnston argue that this early rate elevation might be parsimoniously explained by their finding that in the supercritical case with no deaths the rate of branching events along a uniformly chosen ancestral lineage is increasing in $t \in [0,T]$ (which is a corollary to their Theorem 2.4). 
 
 The two-stage sampling rule
 \begin{itemize}
     \item first sample a random tree (``an adult'') that survives up to time $T$,
     \item then sample an individual  from this tree (``a cell from this adult'') at time $T$
 \end{itemize}seems adequate for the situation discussed in Cheek and Johnston \cite[Section 5]{gworg}. In other modeling  situations, again with a large collection of i.i.d.\ Galton-Watson trees, one may think of a different sampling rule: Choose individuals at time $T$ uniformly from the union of all time $T$ individuals in all of the trees. This makes it more probable that the sampled individuals belong to larger trees, and in fact corresponds to the size-biasing of the random trees at time $T$ (\cite[Section 4]{geba}).
 In the two-stage sampling rule we see the different rate bias \eqref{ratebias}, discussed at the end of \Cref{sec:mainresults}.
  
 As can be seen from \cite[Theorem~1]{CW} (and \Cref{th:CW}), the rate bias \eqref{ratebias} is also absent along the ancestral lineage of an individual whose marker has a prescribed value $s$,  if one considers a situation in which a neutral marker evolves along the trees in small (continuous) mutation steps, and if one takes, for the prescribed  value $s$, the collection of trees so large that one individual at time $T$ has a marker value close to  (ideally: precisely at) $s$.
 
 The sampling rule  that appears in \cite{geiger} (and \Cref{th:2-general}) leads to a rate (and reproduction size) bias along the ancestral lineage that is different from the ones we just discussed. This sampling rule can be defined via i.i.d. real-valued valued neutral markers that are created at each birth and passed to the offspring. The individual sampled at time $T$ (from the tree conditioned to survive up to time $T$) is the one whose marker sequence is the largest in lexicographic order among the individuals that live in the tree at time $T$. This interpretation appears of less biological relevance, except in the pure birth (or cell division) case, where one might think of one single marker that is passed on in each generation to a randomly chosen daughter cell.

\FloatBarrier 

\bibliography{gw} 

\newcommand{\etalchar}[1]{$^{#1}$}
\begin{thebibliography}{CTAS{\etalchar{+}}19}
\expandafter\ifx\csname url\endcsname\relax
  \def\url#1{\texttt{#1}}\fi
\expandafter\ifx\csname doi\endcsname\relax
  \def\doi#1{\burlalt{doi:#1}{http://dx.doi.org/#1}}\fi
\expandafter\ifx\csname urlprefix\endcsname\relax\def\urlprefix{URL }\fi
\expandafter\ifx\csname href\endcsname\relax
  \def\href#1#2{#2}\fi
\expandafter\ifx\csname burlalt\endcsname\relax
  \def\burlalt#1#2{\href{#2}{#1}}\fi

\bibitem[AN72]{AN}
K.~B. Athreya and P.~E. Ney.
\newblock {\em The Galton-Watson Process}.
\newblock Springer Berlin Heidelberg, Berlin, Heidelberg, 1972.
\newblock \doi{10.1007/978-3-642-65371-1}.

\bibitem[BH48]{BH}
R.~Bellman and T.~E. Harris.
\newblock On the theory of age-dependent stochastic branching processes.
\newblock {\em Proceedings of the National Academy of Sciences},
  34(12):601--604, 1948.
\newblock \doi{10.1073/pnas.34.12.601}.

\bibitem[CJ23]{gworg}
D.~Cheek and S.~G.~G. Johnston.
\newblock Ancestral reproductive bias in branching processes.
\newblock {\em Journal of Mathematical Biology}, 86(5):70, 2023.
\newblock \doi{10.1007/s00285-023-01907-7}.

\bibitem[CRW91]{CW}
B.~Chauvin, A.~Rouault, and A.~Wakolbinger.
\newblock Growing conditioned trees.
\newblock {\em Stochastic Processes and their Applications}, 39(1):117--130,
  1991.
\newblock \doi{10.1016/0304-4149(91)90036-C}.

\bibitem[CTAS{\etalchar{+}}19]{coorens}
T.~H. Coorens, T.~D. Treger, R.~Al-Saadi, L.~Moore, M.~G. Tran, T.~J. Mitchell,
  S.~Tugnait, C.~Thevanesan, M.~D. Young, T.~R. Oliver, et~al.
\newblock Embryonal precursors of {W}ilms tumor.
\newblock {\em Science}, 366(6470):1247--1251, 2019.
\newblock \doi{10.1126/science.aax1323}.

\bibitem[GB03]{geba}
H.-O. Georgii and E.~Baake.
\newblock Supercritical multitype branching processes: the ancestral types of
  typical individuals.
\newblock {\em Advances in Applied Probability}, 35(4):1090--1110, 2003.
\newblock \doi{10.1239/aap/1067436336}.

\bibitem[Gei99]{geiger}
J.~Geiger.
\newblock Elementary new proofs of classical limit theorems for
  {G}alton-{W}atson processes.
\newblock {\em Journal of Applied Probability}, 36(2):301--309, 1999.
\newblock \doi{10.1239/jap/1032374454}.

\bibitem[JN96]{jn}
P.~Jagers and O.~Nerman.
\newblock The asymptotic composition of supercritical, multi-type branching
  populations.
\newblock {\em S{\'e}minaire de probabilit{\'e}s de Strasbourg}, 30:40--54,
  1996.
\newblock \doi{10.1007/BFb0094640}.

\bibitem[KLPP97]{klpp}
T.~Kurtz, R.~Lyons, R.~Pemantle, and Y.~Peres.
\newblock {\em A conceptual proof of the {K}esten-{S}tigum theorem for
  multi-type branching processes}, pages 181--185.
\newblock Springer New York, New York, NY, 1997.
\newblock \doi{10.1007/978-1-4612-1862-3_14}.

\bibitem[LPP95]{lpp}
R.~Lyons, R.~Pemantle, and Y.~Peres.
\newblock Conceptual proofs of {L} log {L} criteria for mean behavior of
  branching processes.
\newblock {\em The Annals of Probability}, 23(3):1125--1138, 1995.
\newblock \doi{10.1214/aop/1176988176}.

\bibitem[PMK{\etalchar{+}}21]{park}
S.~Park, N.~M. Mali, R.~Kim, J.-W. Choi, J.~Lee, J.~Lim, J.~M. Park, J.~W.
  Park, D.~Kim, T.~Kim, et~al.
\newblock Clonal dynamics in early human embryogenesis inferred from somatic
  mutation.
\newblock {\em Nature}, 597(7876):393--397, 2021.
\newblock \doi{10.1038/s41586-021-03786-8}.

\end{thebibliography}
\bibliographystyle{habbrv}
\end{document}